\documentclass{jcmlatex}
\usepackage{subfigure}
\graphicspath{{figure/}}
\begin{document}

\def\JCMvol{xx}
\def\JCMno{x}
\def\JCMyear{200x}
\def\JCMreceived{Month xx, 200x}
\def\JCMrevised{}
\def\JCMaccepted{}
\def\JCMonline{}
\setcounter{page}{1}



\def\cgeq{\raisebox{-1mm}{$\;\stackrel{>}{\sim}\;$}}
\def\cequiv{\raisebox{-1mm}{$\;\stackrel{=}{\sim}\;$}}
\def\cleq{\raisebox{-1mm}{$\;\stackrel{<}{\sim}\;$}}
\def\T{\mathop{\cal T}}
\def\esssup{\mathop{\rm esssup}}
\def\qed{\hfill$\fbox{}$}
\catcode`@=11
\newskip\plaincentering \plaincentering=0pt plus 1000pt minus 1000pt
\def\@plainlign{\tabskip=0pt\everycr={}} 
\def\eqalignno#1{\displ@y \tabskip=\plaincentering
  \halign to\displaywidth{\hfil$\@plainlign\displaystyle{##}$\tabskip=0pt
    &$\@plainlign\displaystyle{{}}##$\hfil\tabskip=\plaincentering
    &\llap{$\hbox{\rm\@plainlign##}$}\tabskip=0pt\crcr
    #1\crcr}}



\markboth{W.H.  DENG,  AND M.H. CHEN}
 {Efficient Numerical Algorithms for three-dimensional fractional partial differential equations}

\title{EFFICIENT NUMERICAL ALGORITHMS FOR THREE-DIMENSIONAL FRACTIONAL PARTIAL DIFFERENTIAL EQUATIONS}

\author{Weihua Deng \and Minghua Chen
               \thanks{School of Mathematics and Statistics, Lanzhou University, Lanzhou 730000, China\\
Email: dengwh@lzu.edu.cn, chenmh09@163.com}}

\maketitle

\begin{abstract}
This paper detailedly discusses the locally one-dimensional
numerical methods for efficiently solving the three-dimensional
fractional partial differential equations, including fractional advection diffusion equation and Riesz fractional diffusion equation. The second order finite difference scheme is
used to discretize the space fractional derivative and the
Crank-Nicolson procedure to the time derivative. We theoretically
prove and numerically verify that the presented numerical methods
are unconditionally stable and second order convergent in both space
and time directions. In particular, for the Riesz fractional diffusion equation, the idea of reducing the splitting error is used to further improve the algorithm, and the unconditional stability and convergency are also strictly proved and numerically verified for the improved scheme.

\end{abstract}



\section{Introduction}
The history of fractional calculus can goes back to more than three
hundred years ago \cite{Kenneth:93}, almost the same as classical
calculus. Nowadays it has become more and more popular among various
scientific fields, covering anomalous diffusion, materials and
mechanical, signal processing and systems identification, control
and robotics, rheology, fluid flow, signal processing, and
electrical networks et al \cite{Metzler:00}. Meanwhile, the diverse
fractional partial differential equations (fractional PDEs), as
models, appear naturally in the corresponding field.

There are already some important progress for numerically solving
the fractional PDEs. The methods used for classical PDEs are well
extended to fractional PDEs, for example, the finite difference
method
\cite{Liu:09,Sousa:12,Charles:07,Tian:12},
finite element method \cite{Deng:08,Ervin:05}, and spectral method
\cite{Xu:09}. However, almost all of them concentrate on one or two
dimensional problems.  There are already some good developments for
realizing the operator splitting (locally one dimension) to solve
the classical PDEs. This paper focuses on extending the alternating
direction implicit (ADI) methods to the three-dimensional fractional
PDEs, and improving their efficiency.

The Peaceman and Rachford alternating direction implicit method
(PR-ADI) \cite{Peaceman:55} works well for two-dimensional problems. But it can not be
extended to higher dimensional problems. Douglas type alternating
direction implicit methods (D-ADI)
\cite{Douglas:55,Douglas:64,Douglas:01} are valid for any
dimensional equations. And PR-ADI and D-ADI are equivalent in two
dimensional problems. In this paper, we consider the following three-dimensional
 fractional advection diffusion equation,
\begin{equation}\label{1.1}
\begin{split}
  \frac{\partial u(x,y,z,t) }{\partial t}=&d_1^x\, _{x_L}\!D_x^{\alpha}u(x,y,z,t)+d_2^x\,  _{x}D_{x_R}^{\alpha}u(x,y,z,t) \\
  &+d_1^y \,_{y_L}\!D_y^{\beta}u(x,y,z,t)+d_2^y \,  _{y}D_{y_R}^{\beta}u(x,y,z,t)\\
  &+d_1^z\,   _{z_L}\!D_{z}^{\gamma}u(x,y,z,t)+d_2^z\,  _{z}D_{z_R}^{\gamma}u(x,y,z,t)\\
  &+\kappa_x \frac{\partial u(x,y,z,t)}{ \partial x}  +\kappa_y\frac{\partial u(x,y,z,t)}{ \partial y}+\kappa_z \frac{\partial u(x,y,z,t)}{ \partial z}\\
  &+f(x,y,z,t),
\end{split}
\end{equation}
and the Riesz fractional diffusion equation
\begin{equation*}
\begin{array}{ll}
 \displaystyle \frac{\partial u(x,y,z,t) }{\partial t}=&d_1^x\big(\, _{x_L}\!D_x^{\alpha}u(x,y,z,t)+\, _{x}D_{x_R}^{\alpha}u(x,y,z,t)\big) \\
  & \displaystyle +d_1^y \big( _{y_L}\!D_y^{\beta}u(x,y,z,t)+  \,_{y}D_{y_R}^{\beta}u(x,y,z,t) \big)\\
  & \displaystyle +d_1^z\big(  _{z_L}\!D_{z}^{\gamma}u(x,y,z,t)+  \,_{z}D_{z_R}^{\gamma}u(x,y,z,t) \big)\\
    & \displaystyle +f(x,y,z,t),
  \end{array} \eqno{(1.1^\prime)}
\end{equation*}
both with the initial condition
\begin{equation}\label{1.2}
u(x,y,z,0)=u_0(x,y,z) ~~~ {\rm for}~~~ (x,y,z) \in \Omega,
\end{equation}
and the Dirichlet boundary condition
\begin{equation}\label{1.3}
  u(x,y,z,t)=0 ~~~ {\rm for}~~~ (x,y,z,t) \in \partial \Omega \times (0,T],
\end{equation}
where $\Omega=(x_L,x_R) \times (y_L,y_R)\times (z_L,z_R)\subset
\mathbb{R}^3$, $0< t \leq T$, and the
fractional orders $1<\alpha,\beta, \gamma <2$; and
$f(x,y,z,t)$ is a forcing function; and all the coefficients
are non-negative constants. The fractional derivatives used in
(\ref{1.1}) and $(1.1^\prime)$ are defined as, for $1<\mu <2$,
\begin{equation}\label{1.4}
  _{x_L}D_x^\mu u(x)=\frac{1}{\Gamma(2-\mu)}\frac{\partial^2}{\partial x^2}
 \int_{x_L}^x{\left(x-\xi\right)^{1-\mu}}{u(\xi)}d\xi
\end{equation}
and
\begin{equation}\label{1.5}
 _{x}D_{x_R}^\mu u(x)=\frac{1}{\Gamma (2-\mu)}\frac{\partial^2}{\partial x^2}
  \int_x^{x_R}{(\xi-x)^{1-\mu}}{u(\xi)}d\xi.
\end{equation}
From the viewpoint of conversation law, the advection term in the advection diffusion equation should be first order classical derivative, and the fractional derivative corresponding to the diffusion term should be Riemann-Liouville one.

For the two-dimensional case of (\ref{1.1})-(\ref{1.3}), PR-ADI and
D-ADI are discussed and we show that they are equivalent for
two-dimensional equations. We use D-ADI for the three-dimensional
 (\ref{1.1})-(\ref{1.3}). The second order finite difference
scheme is used to discretize the space fractional derivative and the
Crank-Nicolson procedure to the time direction. We theoretically
prove and numerically confirm that the given numerical schemes are
unconditionally stable and second order convergent in both space and
time directions. In general, the ADI methods introduce new error
term, called the splitting error, comparing with the original
discretizations. Usually the splitting error term does not affect
the convergent order, but most of the time it lowers the accuracy
seriously. For $(1.1^\prime)$, we use the idea in \cite{Douglas:01} to reduce the
splitting error from $\mathcal{O}(\tau^2)$ to
$\mathcal{O}(\tau^3)$ at reasonable computational cost and
then recover the accuracy of the original discretization, the
improved ADI will be called D-ADI-II. The fractional step (FS)
method is also simply discussed to show that, after a minor
modification to reduce the splitting error from $\mathcal{O}(\tau)$ to $\mathcal{O}(\tau^3)$, it is equivalent to D-ADI-II.

The outline of this paper is as follows. In Section 2, we introduce
the second order finite difference schemes for the left and right
Riemann-Liouville fractional derivatives (\ref{1.4}) and
(\ref{1.5}), and the full discretization schemes of the
one-dimensional and two-dimensional case of (\ref{1.1})-(\ref{1.3}) and
(\ref{1.1})-(\ref{1.3}) itself are detailedly provided. Section 3
discusses improving the accuracy and efficiency of ADI, presents
D-ADI-II for $(1.1^\prime)$, and shows that, after a minor modification of the FS
method, it is equivalent to D-ADI-II. We do the convergence and stability
analysis for the schemes used in this paper in Section 4. The
numerical results are given in Section 5 and we conclude this paper
with some discussions in the last section.



\section{Discretization Schemes}\label{sec:1}
We use fourth subsections to derive the full discretization of
(\ref{1.1}), and the corresponding schemes of ($1.1^\prime$) can be obtained by letting $d_1^x=d_2^x$, $d_1^y=d_2^y$, $d_1^z=d_2^z$, and $\kappa_x=\kappa_y=\kappa_z=0$. The first subsection introduces the second
order finite difference schemes for the left and right
Riemann-Liouville fractional derivatives (\ref{1.4}) and (\ref{1.5})
in a finite interval given in \cite{Chen:12} based on the idea of
\cite{Sousa:12}. The second to fourth subsection present the D-ADI
schemes for the one-dimensional and  two-dimensional case of (\ref{1.1})-(\ref{1.3}) and
(\ref{1.1})-(\ref{1.3}) itself, respectively.


\subsection{Discretizations for the left and right Riemann-Liouville fractional derivatives}
Let the mesh points $x_i=x_L+i\Delta x$,  $0 \leq i \leq {N_x}$, $y_j=y_L+j\Delta y$, $0\leq j \leq {N_y}$, $z_m=z_L+m\Delta z$, $0\leq m \leq {N_z}$
and $t_n=n\tau$, $0\leq n \leq {N_t}$, where
 $\Delta x=(x_R-x_L)/{N_x}$, $\Delta y=(y_R-y_L)/{N_y}$, $\Delta z=(z_R-z_L)/{N_z}$, $\tau=T/{N_t}$,
 i.e., $\Delta x$, $\Delta y$ and $\Delta z$ are the uniform space stepsizes in the corresponding directions, $\tau$ the time stepsize.
 For $\mu\in (1,2)$,  the left and right  Riemann-Liouville space fractional
derivatives (\ref{1.3}) and (\ref{1.4}) have the second-order approximation operators
$\delta'_{\mu,_+x}{u_{i,j,m}^n}$ and $\delta'_{\mu,_-x}{u_{i,j,m}^n}$, respectively, given in a finite
domain \cite{Chen:12,Sousa:12}, where ${u_{i,j,m}^n}$ denotes the
approximated value of $u(x_i,y_j,z_m,t_n)$.

The approximation operator of (\ref{1.4}) is defined by \cite{Chen:12,Sousa:12}
\begin{equation}\label{2.1}
\begin{split}
  \delta'_{\mu,_+x}{u_{i,j,m}^n}:=\frac{1}{\Gamma(4-\mu)(\Delta x)^{\mu}}\sum_{l=0}^{i+1} g_l^{\mu}u_{i-l+1,j,m}^n,
  \end{split}
\end{equation}
and there exists
\begin{equation}\label{2.2}
  _{x_L}\!D_x^{\mu}u(x,y_j,z_m,t_n)\big|_{x=x_i}= \delta'_{\mu,_+x}u(x_i,y_j,z_m,t_n)+\mathcal{O}(\Delta x)^2,
\end{equation}
where
\begin{equation*}
 \delta'_{\mu,_+x}u(x_i,y_j,z_m,t_n)=\frac{1}{\Gamma(4-\mu)(\Delta x)^{\mu}}\sum_{l=0}^{i+1} g_l^{\mu}u(x_{i-l+1},y_j,z_m,t_n),
\end{equation*}
and
\begin{equation}\label{2.3}
g_l^{\mu}=\left\{ \begin{array}
 {l@{\quad } l}
1,&l =0,\\
-4+2^{3-\mu},&l =1,\\
6-2^{5-\mu}+3^{3-\mu},&l =2,\\
(l+1)^{3-\mu}-4l^{3-\mu}+6(l-1)^{3-\mu}\\
~~~~~~~~~~~-4(l-2)^{3-\mu}+(l-3)^{3-\mu},&l \geq 3.\\
 \end{array}
 \right.
\end{equation}
Analogously, the approximation operator of (\ref{1.5}) is  described as \cite{Chen:12}
 \begin{equation}\label{2.4}
  \delta'_{\mu,_-x}{u_{i,j,m}^n}:
                              =\frac{1}{\Gamma(4-\mu)(\Delta x)^{\mu}}\!\!\!\sum_{l=0}^{N_x-i+1}\!\!\!g_l^{\mu}u_{i+l-1,j,m}^n,
\end{equation}
where $g_l^{\mu}$ is defined by (\ref{2.3}), and it holds that
\begin{equation}\label{2.5}
  _{x}D_{x_R}^{\mu}u(x,y_j,z_m,t_n)\big|_{x=x_i}= \delta'_{\mu,_-x}u(x_i,y_j,z_m,t_n)+\mathcal{O}(\Delta x)^2,
\end{equation}
where
\begin{equation*}
 \delta'_{\mu,_-x}u(x_i,y_j,z_m,t_n)=\frac{1}{\Gamma(4-\mu)(\Delta x)^{\mu}}\!\!\!\sum_{l=0}^{N_x-i+1}\!\!\! g_l^{\mu}u(x_{i+l-1},y_j,z_m,t_n).
\end{equation*}

In the   following, we introduce and   list  some discrete operators which
work for the functions of three variables $x$, $y$, and $z$:
\begin{equation}\label{2.6}
\begin{split}
& D'_{\alpha,x}u_{i,j,m}^n\!=\!\frac{u_{i+1,j,m}^{n}-u_{i-1,j,m}^n}{2\Delta x};
  ~~~~\,\qquad\qquad\quad \,D''_{\alpha,x}u_{i,j,m}^n\!=\!\kappa_xD'_{\alpha,x}u_{i,j,m}^n; \\
&
\delta'_{\alpha,_+x}{u_{i,j,m}^n}\!=\!\frac{1}{\Gamma(4-\alpha)(\Delta x)^{\alpha}}\sum_{l=0}^{i+1} g_l^{\alpha}u_{i-l+1,j,m}^n;
 \quad \delta''_{\alpha,_+x}{u_{i,j,m}^n}\!=d_1^x\delta'_{\alpha,_+x}{u_{i,j,m}^n};\\
&
\delta'_{\alpha,_-x}{u_{i,j,m}^n}\!=\!\frac{1}{\Gamma(4-\alpha)(\Delta x)^{\alpha}}
 \!\!\! \!\!\!\sum_{l=0}^{N_x-i+1}\!\!\!g_l^{\alpha}u_{i+l-1,j,m}^n; \quad
\delta''_{\alpha,_-x}{u_{i,j,m}^n}\!=d_2^x\delta'_{\alpha,_-x}{u_{i,j,m}^n}.
\end{split}
\end{equation}
The discrete operators related to the variable $x$ or $y$ in the
above also work for functions of two variables $x$ and $y$, e.g.,
$D'_{\alpha,x}u_{i,j}^n=\frac{u_{i+1,j}^{n}-u_{i-1,j}^n}{2\Delta x}$; \,\,
$\delta'_{\alpha,_+x}{u_{i,j}^n}=\frac{1}{\Gamma(4-\alpha)\Delta  x^{\alpha}}  \sum_{l=0}^{i+1} g_l^{\alpha}u_{i-l+1,j}^n$.

Similarly, it is easy to get the one-dimensional and two-dimensioanl case of (\ref{2.1})-(\ref{2.6}).

\begin{remark}[\cite{Chen:12}] Denoting
$\tilde{U}^n=[u_1^n,u_2^n,\cdots,u_{N_x-1}^n]^{\rm T}$, and
 rewriting (\ref{2.1}) and (\ref{2.4}) as matrix forms
$\delta'_{\alpha,_+x} \tilde{U}^n=\tilde{A}\tilde{U}^n+b_1$ and
$\delta'_{\alpha,_-x} \tilde{U}^n=\tilde{B}\tilde{U}^n+b_2$,
respectively, then there exists $\tilde{A}=\tilde{B}^{\rm T}$.
\end{remark}

\subsection{Numerical scheme for 1D}
Consider the full discretization scheme to the one-dimensional
case of (\ref{1.1}), namely,
\begin{equation}\label{2.7}
\begin{split}
\frac{\partial u(x,t)}{\partial t}=&d_1^x\, _{x_L}\!D_x^{\alpha}u(x,t)+d_2^x\,  _{x}D_{x_R}^{\alpha}u(x,t)+\kappa_x \frac{\partial u(x,t)}{ \partial x} +f(x,t).
\end{split}
\end{equation}
In the time direction, we use the Crank-Nicolson scheme.
 The central difference formula, left fractional approximation operator
(\ref{2.2}), and right fractional approximation operator
(\ref{2.5}) are respectively used to discretize the classical second
order space derivative, left Riemann-Liouville fractional
derivative, and right Riemann-Liouville fractional derivative.
Taking the uniform time step $\tau$ and space step $\Delta x$,
and taking $u_{i}^n$ as the approximated value of $u(x_i,t_n)$ and
$f_i^{n+1/2}=f(x_i,t_{n+1/2})$, where $t_{n+1/2}=(t_n+t_{n+1})/2$,
using the one-dimensioanl case of  (\ref{2.1})-(\ref{2.6}), we can write (\ref{2.7}) as
\begin{equation}\label{2.8}
\begin{split}
  \frac{u(x_i,t_{n+1})-u(x_i,t_{n})}{\tau}=&\frac{1}{2} \Big[d_1^x \delta'_{\alpha,_+x}u(x_i,t_{n+1}) +d_1^x\delta'_{\alpha,_+x}u(x_i,t_{n})
 +d_2^x\delta'_{\alpha,_-x}u(x_i,t_{n+1}) \\
 &+d_2^x\delta'_{\alpha,_-x}u(x_i,t_{n})
     +\kappa_x  D'_{\alpha,x}u(x_i,t_{n+1})+\kappa_xD'_{\alpha,x}u(x_i,t_{n})   \Big]\\
     & +f(x_i,t_{n+1/2})+\mathcal{O}(\tau^2+(\Delta x)^2),
\end{split}
\end{equation}
where
\begin{equation*}
  D'_{\alpha,x}u(x_i,t_{n})=\frac{u(x_{i+1},t_{n})-u(x_{i-1},t_{n})}{2\Delta x}.
\end{equation*}
Multiplying (\ref{2.8}) by $\tau$, we have the following equation
\begin{equation}\label{2.9}
\begin{split}
 & \left[ 1- \frac{\tau }{2}\left(   d_1^x\delta'_{\alpha,_+x}  +  d_2^x\delta'_{\alpha,_-x}
                     + \kappa_xD'_{\alpha,x}\right)   \right]  u(x_i,t_{n+1})\\
  & \quad=\left[ 1+ \frac{\tau }{2}\left(   d_1^x\delta'_{\alpha,_+x}  +  d_2^x\delta'_{\alpha,_-x}
                     + \kappa_xD'_{\alpha,x}\right)\right]u(x_i,t_{n})
         +\tau f(x_i,t_{n+1/2})+ R_i^{n+1},
\end{split}
\end{equation}
where
 \begin{equation}\label{2.10}
 |R_i^{n+1}|\leq \widetilde{c} \tau(\tau^2+(\Delta x)^2).
 \end{equation}
Therefore, the full discretization of
 (\ref{2.7}) has the following form
 \begin{equation}\label{2.11}
\begin{split}
 & \left[ 1-\frac{\tau }{2}\left(   d_1^x\delta'_{\alpha,_+x}  +  d_2^x\delta'_{\alpha,_-x}
                     + \kappa_xD'_{\alpha,x}\right)    \right]   u_i^{n+1}\\
 & \quad=\left[ 1+\frac{\tau }{2}\left(   d_1^x\delta'_{\alpha,_+x}  +  d_2^x\delta'_{\alpha,_-x}
                     + \kappa_xD'_{\alpha,x}\right)\right]u_i^n   +\tau f_i^{n+1/2}.
\end{split}
\end{equation}
We can write (\ref{2.11}) as
 \begin{equation}\label{2.12}
\begin{split}
 &  u_i^{n+1} \!-\! \frac{\tau }{2}\left[\frac{ d_1^x}{\Gamma(4-\alpha)(\Delta x)^{\alpha}}\sum_{l=0}^{i+1} g_l^{\alpha}u_{i-l+1}^{n+1}
   +\frac{ d_2^x}{\Gamma(4-\alpha)(\Delta x)^{\alpha}}
 \!\!\! \!\!\!\sum_{l=0}^{N_x-i+1}\!\!\!g_l^{\alpha}u_{i+l-1}^{n+1}
                     +\frac{\kappa_x}{2\Delta x}\left(u_{i+1}^{n+1}- u_{i-1}^{n+1}\right) \right]     \\
 & =  u_i^{n}\!+\!\frac{\tau }{2}\left[ \frac{ d_1^x}{\Gamma(4-\alpha)(\Delta x)^{\alpha}}\sum_{l=0}^{i+1} g_l^{\alpha}u_{i-l+1}^{n}
   + \frac{d_2^x}{\Gamma(4-\alpha)(\Delta x)^{\alpha}}
 \!\!\! \!\!\!\sum_{l=0}^{N_x-i+1}\!\!\!g_l^{\alpha}u_{i+l-1}^{n}
                     +\frac{ \kappa_x}{2\Delta x}\left(u_{i+1}^{n}- u_{i-1}^{n}\right)\right] \\
  &\quad +\tau f_i^{n+1/2}.
\end{split}
\end{equation}

For the convenience of implementation, we use the matrix form of the grid functions
 \begin{equation*}
 U^{n}=[u_1^n,u_2^n,\ldots,u_{N_x-1}^n]^{\rm T}, ~~F^{n+1/2}=[f_1^{n+1/2},f_2^{n+1/2},\ldots,f_{N_x-1}^{n+1/2}]^{\rm T},
  \end{equation*}
therefore, the finite difference scheme (\ref{2.12}) can be rewritten as
 \begin{equation}\label{2.13}
\begin{split}
 &  \left[I- \frac{\tau }{2}\left(\frac{d_1^x}{\Gamma(4-\alpha)(\Delta x)^{\alpha}}A_{\alpha}
   +\frac{ d_2^x}{\Gamma(4-\alpha)(\Delta x)^{\alpha}}A_{\alpha}^T
                     +\frac{\kappa_x}{2\Delta x}B\right) \right] U^{n+1}    \\
 & \quad= \left[I+\frac{\tau }{2} \left(\frac{ d_1^x}{\Gamma(4-\alpha)(\Delta x)^{\alpha}}A_{\alpha}
   +\frac{ d_2^x}{\Gamma(4-\alpha)(\Delta x)^{\alpha}}A_{\alpha}^T
                     +\frac{ \kappa_x}{2\Delta x}B\right) \right] U^{n}+\tau F^{n+1/2},
\end{split}
\end{equation}
where
\begin{equation}\label{2.14}
A_{\alpha}=\left [ \begin{matrix}
g_1^{\alpha}             &g_0^{\alpha}         &0                         &      \cdots              &      0                  &  0  \\
g_2^{\alpha}             &g_1^{\alpha}         &g_0^{\alpha}              &       0                  &     \cdots              &  0 \\
g_3^{\alpha}             &g_2^{\alpha}         &g_1^{\alpha}              &     g_0^{\alpha}         &     \ddots              & \vdots  \\
\vdots                   &\ddots               &       \ddots             &        \ddots            &      \ddots             & \dots    \\
g_{N_x-2}^{\alpha}       &\ddots               &       \ddots             &        \ddots            &   g_1^{\alpha}          & g_0^{\alpha} \\
g_{N_x-1}^{\alpha}       &g_{N_x-2}^{\alpha}   &    g_{N_x-3}^{\alpha}    &         \cdots           &g_2^{\alpha}             & g_1^{\alpha}
 \end{matrix}
 \right ],  \quad \mbox{and} \quad
 B=\left [ \begin{matrix}
0            &1        &0                         &      \cdots              &      0                  &  0  \\
-1             &0        &1             &       0                  &     \cdots              &  0 \\
0             &-1        &0             &     1         &     \ddots              & \vdots  \\
\vdots                   &\ddots               &       \ddots             &        \ddots            &      \ddots             & \dots    \\
0       &\ddots               &       \ddots             &        \ddots            &  0         & 1 \\
0       &0   &    0   &         \cdots           &-1            & 0
 \end{matrix}
 \right ].
\end{equation}

\subsection{PR-ADI and D-ADI schemes for 2D}

We now examine the full discretization scheme to the two-dimensional
case of (\ref{1.1}), i.e.,

\begin{equation}\label{2.15}
\begin{split}
  \frac{\partial u(x,y,t) }{\partial t}=&d_1^x\, _{x_L}\!D_x^{\alpha}u(x,y,t)+d_2^x\,  _{x}D_{x_R}^{\alpha}u(x,y,t)
  +d_1^y \,_{y_L}\!D_y^{\beta}u(x,y,t)+d_2^y \,  _{y}D_{y_R}^{\beta}u(x,y,t)\\
  &+\kappa_x \frac{\partial u(x,y,t)}{ \partial x}  +\kappa_y\frac{\partial u(x,y,t)}{ \partial y}+f(x,y,t).
\end{split}
\end{equation}

Analogously we still use the Crank-Nicolson scheme to do the discretization in time direction.
Taking $u_{i,j}^n$ as the approximated value of $u(x_i,y_j,t_n)$,
using the two-dimensioanl case of  (\ref{2.1})-(\ref{2.6}), we can write (\ref{2.15}) as
\begin{equation}\label{2.16}
\begin{split}
 &\left [1-\frac{\tau}{2}\left( \delta''_{\alpha,_+x}+  \delta''_{\alpha,_-x} + D''_{\alpha,x}  \right )
 -\frac{\tau}{2}\left( \delta''_{\beta,_+y} +  \delta_{\beta,_-y}^{''} + D''_{\beta,y}   \right)\right ]u(x_i,y_j,t_{n+1}) \\
 &\quad=\left [1+\frac{\tau}{2}\left( \delta''_{\alpha,_+x} +  \delta''_{\alpha,_-x} + D''_{\alpha,x}  \right )
 +\frac{\tau}{2}\left( \delta''_{\beta,_+y} +  \delta''_{\beta,_-y} + D''_{\beta,y} \right )   \right ]u(x_i,y_j,t_{n})\\
 &\qquad+ \tau f(x_i,y_j,t_{n+1/2})
 + R_{i,j}^{n+1},
\end{split}
\end{equation}
where
 \begin{equation}\label{2.17}
 |R_{i,j}^{n+1}|\leq \widetilde{c} \tau(\tau^2+(\Delta x)^2+(\Delta y)^2).
 \end{equation}
Using the notations of (\ref{2.6}),  we further define
\begin{equation*}
\begin{split}
 &\delta_{\alpha,x}:= \delta''_{\alpha,_+x} +  \delta''_{\alpha,_-x} + D''_{\alpha,x};\\
 &\delta_{\beta,y}:=\delta''_{\beta,_+y} +  \delta''_{\beta,_-y} +D''_{\beta,y},
\end{split}
\end{equation*}
thus, the resulting discretization of (\ref{2.15}) can be written as a Crank-Nicolson type finite difference equation
\begin{equation}\label{2.18}
\begin{split}
  \frac{u_{i,j}^{n+1}-u_{i,j}^n}{\tau}
  =\frac{\delta_{\alpha,x}u_{i,j}^{n+1}+\delta_{\alpha,x}u_{i,j}^{n}+\delta_{\beta,y}u_{i,j}^{n+1}+\delta_{\beta,y}u_{i,j}^{n}}{2}+ f_{i,j}^{n+1/2},
\end{split}
\end{equation}
i.e.,
\begin{equation}\label{2.19}
\begin{split}
 &\left [1-\frac{\tau}{2}\left( \delta''_{\alpha,_+x}+  \delta''_{\alpha,_-x} + D''_{\alpha,x}  \right )
 -\frac{\tau}{2}\left( \delta''_{\beta,_+y} +  \delta_{\beta,_-y}^{''} + D''_{\beta,y}   \right)\right ]u_{i,j}^{n+1} \\
 &\quad=\left [1+\frac{\tau}{2}\left( \delta''_{\alpha,_+x} +  \delta''_{\alpha,_-x} + D''_{\alpha,x}  \right )
 +\frac{\tau}{2}\left( \delta''_{\beta,_+y} +  \delta''_{\beta,_-y} + D''_{\beta,y} \right )   \right ]u_{i,j}^{n}+ \tau f_{i,j}^{n+1/2},
\end{split}
\end{equation}
or
\begin{equation}\label{2.20}
 \left (1-\frac{\tau}{2}\delta_{\alpha,x}-\frac{\tau}{2}\delta_{\beta,y}\right )u_{i,j}^{n+1}
 =\left (1+\frac{\tau}{2}\delta_{\alpha,x}+\frac{\tau}{2}\delta_{\beta,y}\right )u_{i,j}^{n}+ \tau f_{i,j}^{n+1/2}.
\end{equation}
The perturbation equation of (\ref{2.20}) is of the form
\begin{equation}\label{2.21}
\begin{split}
 \left (1-\frac{\tau}{2}\delta_{\alpha,x} \right )\left(1-\frac{\tau}{2}\delta_{\beta,y}\right )u_{i,j}^{n+1}
 =\left (1+\frac{\tau}{2}\delta_{\alpha,x}\right)\left(1+\frac{\tau}{2}\delta_{\beta,y}\right )u_{i,j}^{n}+ \tau f_{i,j}^{n+1/2}.
\end{split}
\end{equation}
Comparing (\ref{2.21}) with (\ref{2.20}), the splitting term is
given by
\begin{equation} \label{2.22}
\frac{\tau^2}{4}\delta_{\alpha,x}\delta_{\beta,y}(u_{i,j}^{n+1}-u_{i,j}^{n}),
\end{equation}
since $(u_{i,j}^{n+1}-u_{i,j}^n)$ is an
$\mathcal{O}(\tau)$ term, it implies  that this perturbation
contributes an $\mathcal{O}(\tau^2)$ error component to the truncation error of the  Crank-Nicolson finite difference method (\ref{2.18}).

The system of equations defined by (\ref{2.21}) can be solved by the following systems.

PR-ADI scheme \cite{Peaceman:55}:
\begin{equation}\label{2.23}
\left (1-\frac{\tau}{2}\delta_{\alpha,x} \right )u_{i,j}^{*}
 =\left(1+\frac{\tau}{2}\delta_{\beta,y}\right )u_{i,j}^{n}+ \frac{\tau}{2}
 f_{i,j}^{n+1/2};
\end{equation}
\begin{equation}\label{2.24}
\left(1-\frac{\tau}{2}\delta_{\beta,y}\right )u_{i,j}^{n+1}
=\left (1+\frac{\tau}{2}\delta_{\alpha,x}\right)u_{i,j}^{*}+
\frac{\tau}{2} f_{i,j}^{n+1/2}.
\end{equation}

 D-ADI scheme \cite{Douglas:55,Douglas:64,Douglas:01}:
\begin{equation}\label{2.25}
\left (1-\frac{\tau}{2}\delta_{\alpha,x} \right )u_{i,j}^{*}
 =\left(1+\frac{\tau}{2}\delta_{\alpha,x}+\tau\delta_{\beta,y}\right )u_{i,j}^{n}+\tau f_{i,j}^{n+1/2};
\end{equation}
\begin{equation}\label{2.26}
\left(1-\frac{\tau}{2}\delta_{\beta,y}\right )u_{i,j}^{n+1}
=u_{i,j}^{*}- \frac{\tau}{2}\delta_{\beta,y}u_{i,j}^{n}.
\end{equation}
Take
\begin{equation*}
\begin{split}
&\mathbf{U}^{n}=[u_{1,1}^n,u_{2,1}^n,\ldots,u_{N_x-1,1}^n,u_{1,2}^n,u_{2,2}^n,\dots,u_{N_x-1,2}^n,\ldots,u_{1,N_y-1}^n,u_{2,N_y-1}^n,\ldots,u_{N_x-1,N_y-1}^n]^T,\\
&\mathbf{F}^{n}=[f_{1,1}^n,f_{2,1}^n,\ldots,f_{N_x-1,1}^n,f_{1,2}^n,f_{2,2}^n,\dots,f_{N_x-1,2}^n,\ldots,f_{1,N_y-1}^n,f_{2,N_y-1}^n,\ldots,f_{N_x-1,N_y-1}^n]^T,\\
\end{split}
\end{equation*}
and denote
\begin{equation}\label{2.27}
\begin{split}
&\mathcal{B}_x =\frac{ d_1^x\tau}{2\Gamma(4-\alpha)(\Delta x)^{\alpha}}I \otimes A_{\alpha}
               +\frac{  d_2^x\tau}{2\Gamma(4-\alpha)(\Delta x)^{\alpha}}I \otimes A_{\alpha}^T +\frac{\kappa_x\tau }{4\Delta x}I \otimes B,\\
&\mathcal{B}_y =\frac{ d_1^y\tau}{2\Gamma(4-\beta)(\Delta y)^{\beta}} A_{\beta}  \otimes I
               +\frac{  d_2^y\tau}{2\Gamma(4-\beta)(\Delta y)^{\beta}}  A_{\beta}^T \otimes I +\frac{\kappa_y\tau }{4\Delta y}  B \otimes I,
\end{split}
 \end{equation}
where $I$  denotes the unit matrix and  the symbol $\otimes$ the Kronecker product \cite{Laub:05},
 the matrixes  $A_{\alpha}$, $A_{\beta}$ and $B$ are defined by (\ref{2.14}) corresponding to $\alpha$ and $\beta$, respectively. Thus, the finite difference scheme (\ref{2.21}) has the following form
\begin{equation}\label{2.28}
  (I-\mathcal{B}_x )(I-\mathcal{B}_y )\mathbf{U}^{n+1}= (I+\mathcal{B}_x )(I+\mathcal{B}_y )\mathbf{U}^{n}+\tau \mathbf{F}^{n+1/2}.
\end{equation}

\begin{remark}
The schemes (\ref{2.23})-(\ref{2.24}) and (\ref{2.25})-(\ref{2.26})
are equivalent, since both of them come from (\ref{2.21}).
\end{remark}

\subsection{D-ADI scheme for 3D}
Using the notations of (\ref{2.1})-(\ref{2.6}), we can write (\ref{1.1}) as the following form
\begin{equation}\label{2.29}
\begin{split}
 &\left (1-\frac{\tau}{2}\delta_{\alpha,x}-\frac{\tau}{2}\delta_{\beta,y}-\frac{\tau}{2}\delta_{\gamma,z}\right )u(x_i,y_j,z_m,t_{n+1})\\
 &\quad=\left (1+\frac{\tau}{2}\delta_{\alpha,x}+\frac{\tau}{2}\delta_{\beta,y}+\frac{\tau}{2}\delta_{\gamma,z}\right )u(x_i,y_j,z_m,t_{n})
 +  \tau f(x_i,y_j,z_m,t_{n+1/2})+ R_{i,j,m}^{n+1},
\end{split}
\end{equation}
where
\begin{equation}\label{2.30}
\begin{split}
 &\delta_{\alpha,x}:= \delta''_{\alpha,_+x} +  \delta''_{\alpha,_-x} + D''_{\alpha,x};\\
 &\delta_{\beta,y}:=\delta''_{\beta,_+y} +  \delta''_{\beta,_-y} + D''_{\beta,y};\\
 &\delta_{\gamma,z}:=\delta''_{\gamma,_+z} +  \delta''_{\gamma,_-z} +
 D''_{\gamma,z},
\end{split}
\end{equation}
and
\begin{equation}\label{2.31}
 |R_{i,j,m}^{n+1}|\leq \widetilde{c} \tau(\tau^2+(\Delta x)^2+(\Delta y)^2+(\Delta z)^2).
 \end{equation}

 Similarly, the full  discretization  scheme of   (\ref{1.1}) can be written as
\begin{equation}\label{2.32}
\begin{split}
 &\left (1-\frac{\tau}{2}\delta_{\alpha,x}-\frac{\tau}{2}\delta_{\beta,y}-\frac{\tau}{2}\delta_{\gamma,z}\right )u_{i,j,m}^{n+1}\\
 &\quad=\left (1+\frac{\tau}{2}\delta_{\alpha,x}+\frac{\tau}{2}\delta_{\beta,y}+\frac{\tau}{2}\delta_{\gamma,z}\right )u_{i,j,m}^{n}+  \tau f_{i,j,m}^{n+1/2}.
    \end{split}
\end{equation}
The perturbation equation of (\ref{2.32}) is of the form
\begin{equation}\label{2.33}
\begin{split}
 &  \left(1-\frac{\tau}{2}\delta_{\alpha,x}\right)      \left(1-\frac{\tau}{2}\delta_{\beta,y}\right)
    \left(1-\frac{\tau}{2} \delta_{\gamma,z}\right)  u_{i,j,m}^{n+1}\\
 &\quad =  \left (1+\frac{\tau}{2}\delta_{\alpha,x}\right) \left(1+\frac{\tau}{2}\delta_{\beta,y}\right )
          \left(1+\frac{\tau}{2}\delta_{\gamma,z}\right) u_{i,j,m}^{n}
 +f_{i,j,m}^{n+1/2}\tau.
\end{split}
\end{equation}
The scheme (\ref{2.33}) differs from (\ref{2.32}) by the perturbation term
$$ \frac{(\Delta
  t)^2}{4}(\delta_{\alpha,x}\delta_{\beta,y}+\delta_{\alpha,x}\delta_{\gamma,z}+\delta_{\beta,y}\delta_{\gamma,z})(u_{i,j,m}^{k+1}-u_{i,j,m}^n)
  - \frac{(\tau)^3}{8}\delta_{\alpha,x}\delta_{\beta,y}\delta_{\gamma,z}\left(u_{i,j,m}^{k+1}-u_{i,j,m}^n\right).$$

The system of equations defined by (\ref{2.33}) can be solved by the
D-ADI scheme \cite{Douglas:55,Douglas:64,Douglas:01}:
\begin{equation}\label{2.34}
\left (1-\frac{\tau}{2}\delta_{\alpha,x} \right )u_{i,j,m}^{n,1}
   =\left(1+\frac{\tau}{2}\delta_{\alpha,x}+\tau\delta_{\beta,y}+\tau\delta_{\gamma,z}\right )u_{i,j,m}^{n}+\tau f_{i,j,m}^{n+1/2};
\end{equation}
\begin{equation}\label{2.35}
\left(1-\frac{\tau}{2}\delta_{\beta,y}\right )u_{i,j,m}^{n,2}=u_{i,j,m}^{n,1}- \frac{\tau}{2}\delta_{\beta,y}u_{i,j,m}^{n};
\end{equation}
\begin{equation}\label{2.36}
\left(1-\frac{\tau}{2}\delta_{\gamma,z}\right )u_{i,j,m}^{n+1}=u_{i,j,m}^{n,2}- \frac{\tau}{2}\delta_{\gamma,z}u_{i,j,m}^{n}.
\end{equation}
Similarly, we  suppose
\begin{equation*}
\begin{split}
 \mathbf{\widetilde{U}}^{n}=
     &\left[u_{1,1,1}^n,u_{2,1,1}^n,\ldots,u_{N_x-1,1,1}^n,             u_{1,2,1}^n,u_{2,2,1}^n,\dots,u_{N_x-1,2,1}^n,\ldots,\right.\\
     &~~ u_{1,N_y-1,1}^n,u_{2,N_y-1,1}^n,\ldots,u_{N_x-1,N_y-1,1}^n,\rightarrow \\
     &~~u_{1,1,2}^n,u_{2,1,2}^n,\ldots,u_{N_x-1,1,2}^n,                 u_{1,2,2}^n,u_{2,2,2}^n,\dots,u_{N_x-1,2,2}^n,\ldots,\\
     & ~~ u_{1,N_y-1,2}^n,u_{2,N_y-1,2}^n,\ldots,u_{N_x-1,N_y-1,2}^n,\rightarrow\\
     &~~\ldots u_{1,1,N_z-1}^n,u_{2,1,N_z-1}^n,\ldots,u_{N_x-1,1,N_z-1}^n,     u_{1,2,N_z-1}^n,u_{2,2,N_z-1}^n,\dots,u_{N_x-1,2,N_z-1}^n,\ldots,\rightarrow\\
     &~ \left. u_{1,N_y-1,N_z-1}^n,u_{2,N_y-1,N_z-1}^n,\ldots,u_{N_x-1,N_y-1,N_z-1}^n\right]^T,
\end{split}
\end{equation*}
and
\begin{equation}\label{2.37}
\begin{split}
&\mathcal{A}_x =\frac{ d_1^x\tau}{2\Gamma(4-\alpha)(\Delta x)^{\alpha}} I \otimes I \otimes A_{\alpha}
 +\frac{  d_2^x\tau}{2\Gamma(4-\alpha)(\Delta x)^{\alpha}}I \otimes I \otimes A_{\alpha}^T +\frac{\kappa_x\tau }{4\Delta x}I \otimes I \otimes B,\\
&\mathcal{A}_y =\frac{ d_1^y\tau}{2\Gamma(4-\beta)(\Delta y)^{\beta}} I \otimes A_{\beta}  \otimes I
+\frac{  d_2^y\tau}{2\Gamma(4-\beta)(\Delta y)^{\beta}} I \otimes A_{\beta}^T \otimes I +\frac{\kappa_y\tau }{4\Delta y} I \otimes B \otimes I,\\
&\mathcal{A}_z =\frac{ d_1^z\tau}{2\Gamma(4-\beta)(\Delta z)^{\gamma}} A_{\gamma}  \otimes I \otimes I
+\frac{  d_2^z\tau}{2\Gamma(4-\gamma)(\Delta z)^{\gamma}}  A_{\gamma}^T \otimes I \otimes I +\frac{\kappa_z\tau }{4\Delta z}  B \otimes I \otimes I,
\end{split}
 \end{equation}
where $I$  denotes the unit matrix and  the symbol $\otimes$ the Kronecker product \cite{Laub:05},
 the matrixes  $A_{\alpha}$, $A_{\beta}$, $A_{\gamma}$ and $B$ are defined by (\ref{2.14}) corresponding to $\alpha$, $\beta$ and $\gamma$, respectively. Thus, the finite difference scheme (\ref{2.33}) has the following form
\begin{equation}\label{2.38}
  (I-\mathcal{A}_x )(I-\mathcal{A}_y )(I-\mathcal{A}_z )\mathbf{\widetilde{U}}^{n+1}= (I+\mathcal{A}_x )(I+\mathcal{A}_y )(I+\mathcal{A}_z )\mathbf{\widetilde{U}}^{n}+\tau \mathbf{\widetilde{F}}^{n+1/2}.
\end{equation}

The corresponding procedure is executed as follows:
\begin{description}
 \item[(1)] First for every fixed  $z=z_m$ $ (m=1,\ldots,{N_z}-1)$,
  and each fixed  $y=y_j$ $ (j=1,\ldots,{N_y}-1)$,
 solving a set of ${N_x}-1$ equations defined by (\ref{2.34}) at the mesh points $ x_k,k=1,\ldots,{N_x}-1$, to get
 $u_{k,j,m}^{n,1}$;
 \item[(2)] Next alternating the spatial direction, and for each fixed  $x=x_i$ $(i=1,\ldots,{N_x}-1)$,
 and each fixed  $z=z_m$ $(m=1,\ldots,{N_z}-1)$,
  solving a set of ${N_y}-1$  equations defined by (\ref{2.35}) at the points $y_k,k=1,\ldots,{N_y}-1$, to obtain $u_{i,k,m}^{n,2}$.
   \item[(3)] At last alternating the spatial direction again, and for each fixed  $y=y_j$ $(j=1,\ldots,{N_y}-1)$,
 and each fixed  $x=x_i$ $(i=1,\ldots,{N_x}-1)$,
  solving a set of ${N_z}-1$  equations defined by (\ref{2.36}) at the points $z_k,k=1,\ldots,{N_z}-1$, to gain $u_{i,j,k}^{n+1}$.
\end{description}

 \section{Improved Accuracy for D-ADI and FS Procedures}
This section shows that the idea of improving the accuracy of D-ADI
and FS procedures \cite{Douglas:01} also works well when used to
solve Riesz fractional diffusion equation $(1.1^\prime)$. For the simpleness to illustrate and discuss
this, we focuses on two-dimensional case of $(1.1^\prime)$. It is natural to
extend higher dimensions. The reason why we abruptly discuss FS
procedure here is because we want to show FS method is equivalent to
D-ADI after some minor modifications even when solving fractional
PDEs.


\subsection{Correction term for the D-ADI method }
The D-ADI scheme of (\ref{2.15}) introduces the splitting error term
(\ref{2.22}). Even though it is still with the order
$\mathcal{O}(\tau^2)$, sometimes it will seriously impair the
accuracy, see Table {\ref{01}}. If we add
\begin{equation} \label{3.1}
\frac{\tau^2}{4}\delta_{\alpha,x}\delta_{\beta,y}(u_{i,j}^n-u_{i,j}^{n-1})
\end{equation}
to the right hand side of (\ref{2.25}), then the new D-ADI, called
D-ADI-II, will have the splitting error
\begin{equation} \label{3.2}
\frac{\tau^2}{4}\delta_{\alpha,x}\delta_{\beta,y}(u_{i,j}^{n+1}-2u_{i,j}^n+u_{i,j}^{n-1}),
\end{equation}
then the splitting error is reduced to $\mathcal{O}(\tau^3)$.
The D-ADI-II is obviously two-step method; $u_{i,j}^1$ can be
obtained by D-ADI first, then initiate D-ADI-II.
\subsection{Correction term for the FS method }
The original FS method for (\ref{2.15}) should be
\begin{equation} \label{3.3}
 \begin{split}
  \frac{u_{i,j}^{n,1}-u_{i,j}^n}{\tau}&=\frac{1}{2}\delta_{\alpha,x} (u_{i,j}^{n,1}+u_{i,j}^n)+f^{n+\frac{1}{2}};\\
  \frac{u_{i,j}^{n+1}-u_{i,j}^{n,1}}{\tau}&=\frac{1}{2}\delta_{\beta,y}
  (u_{i,j}^{n+1}+u_{i,j}^n),
   \end{split}
\end{equation}
 it can be written as
 \begin{equation}\label{3.4}
\begin{split}
     (1-\frac{\tau}{2}\delta_{\alpha,x})u_{i,j}^{n,1}&=(1+\frac{\tau}{2}\delta_{\alpha,x})u_{i,j}^n+\tau f^{n+\frac{1}{2}};\\
     (1-\frac{\tau}{2}\delta_{\beta,y})u_{i,j}^{n+1}&=u_{i,j}^{n,1}+\frac{\tau}{2}\delta_{\beta,y}u^{n}.
\end{split}
\end{equation}
Comparing (\ref{3.3}) with (\ref{2.20}), the splitting term is
given by
\begin{equation} \label{3.5}
\frac{\tau^2}{4}\delta_{\alpha,x}\delta_{\beta,y}(u_{i,j}^{n+1}+u_{i,j}^{n}),
\end{equation}
which is of the order $\mathcal{O}(\tau)$  error component to the truncation error of the  Crank-Nicolson finite difference method (\ref{2.18}). However, if we add
\begin{equation} \label{3.6}
\frac{\tau^2}{4}\delta_{\alpha,x}\delta_{\beta,y}(3u_{i,j}^n-u_{i,j}^{n-1})
\end{equation}
to the right hand side of the first equation of (\ref{3.4}), then
we get the new FS method, called FS-II, with the splitting error
(\ref{3.2}), i.e.,  the splitting error is reduced to $\mathcal{O}(\tau^3)$ .

 The FS-II is equivalent to D-ADI-II,
 since both of them come from the following perturbation equation
 \begin{equation*}
\begin{split}
 &\left (1-\frac{\tau}{2}\delta_{\alpha,x} \right )\left(1-\frac{\tau}{2}\delta_{\beta,y}\right )u_{i,j}^{n+1}\\
 &\quad =\left (1+\frac{\tau}{2}\delta_{\alpha,x}\right)\left(1+\frac{\tau}{2}\delta_{\beta,y}\right )u_{i,j}^{n}
 +\frac{\tau^2}{4}\delta_{\alpha,x}\delta_{\beta,y}(u_{i,j}^{n}-u_{i,j}^{n-1})+ \tau f_{i,j}^{n+1/2},
\end{split}
\end{equation*}
i.e.,
 \begin{equation}\label{3.7}
\begin{split}
\frac{u_{i,j}^{n+1}-u_{i,j}^{n}}{\tau}=\frac{1}{2}\left(\delta_{\alpha,x}+ \delta_{\beta,y}  \right)(u_{i,j}^{n+1}+u_{i,j}^{n})
 -\frac{\tau}{4}\delta_{\alpha,x}\delta_{\beta,y}(u_{i,j}^{n+1}-2u_{i,j}^{n}+u_{i,j}^{n-1})+  f_{i,j}^{n+1/2}.
\end{split}
\end{equation}

  \subsection{Accuracy and efficiency of the D-ADI, D-ADI-II, and FS-II methods}

To check the accuracy and efficiency of the D-ADI, D-ADI-II, and FS-II schemes,
we consider the two-dimensional case of the Riesz fractional equation $(1.1^\prime)$, on a finite domain $ 0<  x< 1,\,0<  y< 1$, $0< t \leq 1$, with  the coefficients
$
  d_1^x=d_1^y=1,
$
and the initial condition $u(x,y,0)=sin((2x)^4)sin((2-2x)^4)sin((2y)^2)sin((2-2y)^2)$ and the
Dirichlet boundary conditions on the rectangle in the form $
u(0,y,t)=u(x,0,t)=0$ and $ u(1,y,t)=u(x,1,t)=0$
 for all $ t\geq 0$. The exact solution to this two-dimensional Riesz fractioanl diffusion equation is
 \begin{equation*}
\begin{split}
u(x,y,t)=e^{-t}sin((2x)^4)sin((2-2x)^4)sin((2y)^4)sin((2-2y)^4).
  \end{split}
  \end{equation*}
By the algorithm given in \cite{Deng:07} and above conditions, it is easy to obtain the forcing function $f(x,y,t)$
at anywhere of the considered rectangle domain with any desired accuracy.

%


\begin{table}  [h]  \fontsize{9.5pt}{12pt}  \selectfont
\begin{center}
\caption{The performance of the D-ADI, D-ADI-II, and FS-II methods with
$\Delta x =\Delta y=1/100$, and the maximum errors (\ref{5.1}). }\vspace{5pt}

    \begin{tabular*}{\linewidth}{@{\extracolsep{\fill}}*{5}{c|c|c|c|c }}          \hline
         $\alpha=1.9,~\beta=1.9$ & $\tau = 10\Delta x$     & $\tau=5\Delta x$     & $\tau=5\Delta x/2$     & $\tau=\Delta x $     \\\hline
         ~~D-ADI                 &3.3496e-02             & 7.6895e-03         & 2.0624e-03        &6.0826e-04           \\\hline
         ~~FS-II               &2.2638e-03            &1.7249e-04        & 2.7765e-04       &3.3166e-04               \\\hline
         ~~D-ADI-II             &2.2638e-03           &1.7249e-04        &2.7765e-04       &3.3166e-04            \\\hline
\end{tabular*}{\label{01}}
\end{center}
\end{table}

From Table {\ref{01}}, we further verify that the
D-ADI-II is equivalent to the FS-II method, and they may reduce the
perturbation error of D-ADI procedure and improve the accuracy.

\section{Convergence and Stability Analysis}

In the following, we denote by $H$ the symmetric (or hermitian) part of $A$ if $A$ is real (or complex) matrix,
and  $||\cdot||$ the matrix 2-norm.

\begin{lemma}\cite{Chen:12,Sousa:12}\label{lemma1}
The coefficients $g_l^{\mu}$, $\mu \in(1,2)$ defined in (\ref{2.3}) satisfy the following properties
\begin{equation*}\label{3.27}
  \begin{split}
&(1) ~~g_0^{\mu}=1, g_1^{\mu}=-4+2^{3-\mu}<0,g_2^{\mu}=6-2^{5-\mu}+3^{3-\mu};\\
&(2) ~~1 \geq g_0^{\mu} \geq g_3^{\mu} \geq g_4^{\mu} \geq \ldots \geq 0;\\
&(3) ~~\sum\limits_{l=0}^{\infty}g_l^{\mu}=0, ~~  \sum\limits_{l=0}^{m}g_l^{\mu}<0, m \geq 2.
  \end{split}
\end{equation*}
\end{lemma}

\begin{lemma}\cite[p.\,28]{Quarteroni:07}\label{lemma2}
A real matrix $A$ of order n is positive definite  if and only if  its symmetric part $H=\frac{A+A^T}{2}$ is positive definite;
$H$ is positive definite if and only if the eigenvalues of $H$ are positive.
\end{lemma}

\begin{lemma}\cite[p.\,184]{Quarteroni:07}\label{lemma3}
If $A \in \mathbb{C}^{n \times n}$, let $H=\frac{A+A^H}{2}$ be the hermitian part of $A$,  then for any eigenvalue $\lambda$ of  $A$,
the real part $\Re(\lambda(A))$ satisfies
\begin{equation*}
  \lambda_{min}(H) \leq \Re(\lambda(A)) \leq \lambda_{max}(H),
\end{equation*}
where $\lambda_{min}(H)$ and $\lambda_{max}(H)$ are the minimum and maximum of the eigenvalues of $H$, respectively.
\end{lemma}

\begin{theorem}\label{theorem1}
Let  matrix $A_{\alpha}$ be defined by (\ref{2.14}), where $\alpha \in (1,2)$, then for any eigenvalue $\lambda$ of $A_{\alpha}$, the real part  $\Re(\lambda(A_{\alpha}))<0$,
and the matrix $A_{\alpha}$ is negative definite.
Moreover, $\Re(\lambda(d_1A_{\alpha}+d_2A_{\alpha}^T))<0$, where
$d_1,d_2\geq 0$, $d_1^2+d_2^2\neq 0$.
\end{theorem}

\begin{proof}
  Let $H=\frac{A_{\alpha}+A_{\alpha}^T}{2}$, by (\ref{2.14})  we know
  \begin{equation}\label{4.1}
H\equiv (h_{i,j})=\frac{1}{2}\left [ \begin{matrix}
2g_1^{\alpha}            &g_0^{\alpha}+g_2^{\alpha}&g_3^{\alpha}             &      \cdots              &  g_{N_x-2}^{\alpha}     &  g_{N_x-1}^{\alpha}  \\
g_0^{\alpha}+g_2^{\alpha}&        2g_1^{\alpha}    &g_0^{\alpha}+g_2^{\alpha}&       g_3^{\alpha}       &     \cdots              &  g_{N_x-2}^{\alpha} \\
g_3^{\alpha}             &g_0^{\alpha}+g_2^{\alpha}&2g_1^{\alpha}            & g_0^{\alpha}+g_2^{\alpha}&     \ddots              & \vdots  \\
\vdots                   &          \ddots         &       \ddots            &        \ddots            &      \ddots             &  g_3^{\alpha}  \\
g_{N_x-2}^{\alpha}       &          \ddots         &       \ddots            &        \ddots            &   2g_1^{\alpha}         & g_0^{\alpha}+g_2^{\alpha} \\
g_{N_x-1}^{\alpha}       &    g_{N_x-2}^{\alpha}   &    g_3^{\alpha}         &         \cdots           &g_0^{\alpha}+g_2^{\alpha}& 2g_1^{\alpha}
 \end{matrix}
 \right ].
\end{equation}

From Lemma \ref{lemma1}, it is easy to check that $g_0^{\alpha}+g_2^{\alpha}>0$,  and the sum of the  absolute value of the off-diagonal entries on the row $i$ of matrix $H$ is given by
\begin{equation*}
  r_i= \sum\limits_{j=1,j\neq i}^{N_x-1}\!\!\!\! |h_{i,j}| < -g_1^{\alpha}.
\end{equation*}

According to the Greschgorin theorem  \cite[p.\,135]{Isaacson:66}, the eigenvalues of the matrix $H$ are in the disks centered at $h_{i,i}$, with radius
$r_i$, i.e.,  the eigenvalues $\lambda$ of the matrix  $H$ satisfies
\begin{equation*}
  |\lambda -h_{i,i} | =|\lambda -g_1^{\alpha} |\leq r_i,
\end{equation*}
it implies that $\lambda(H)<0$. From Lemma \ref{lemma2} and \ref{lemma3},
we obtain that $\Re(\lambda(A_{\alpha}))<0$ and $A$ is negative definite.
Taking
 \begin{equation*}
   \widetilde{H}=\frac{(d_1A_{\alpha}+d_2A_{\alpha}^T)+(d_1A_{\alpha}+d_2A_{\alpha}^T)^T}{2}=(d_1+d_2)\frac{A_{\alpha}+A_{\alpha}^T}{2}=(d_1+d_2)H,
 \end{equation*}
 similarly, we can prove  $\Re(\lambda(d_1A_{\alpha}+d_2A_{\alpha}^T))<0$.
\end{proof}

In the following, we list some properties of the Kronecker Product.
\begin{lemma}\cite[p.\,140]{Laub:05}\label{lemma4}
Let $A \in \mathbb{R}^{m\times n}$, $B \in \mathbb{R}^{r\times s}$, $C \in \mathbb{R}^{n\times p}$, and $D \in \mathbb{R}^{s\times t}$.
Then
\begin{equation*}
  (A \otimes B)(C \otimes D)=AC \otimes  BD \quad (\in \mathbb{R}^{mr\times pt}).
\end{equation*}
\end{lemma}

\begin{lemma}\cite[p.\,140]{Laub:05}\label{lemma5}
For all $A$ and $B$, $(A \otimes B)^T=A^T\otimes B^T$.
\end{lemma}

\begin{lemma}\cite[p.\,141]{Laub:05}\label{lemma6}
Let $A \in \mathbb{R}^{n\times n}$ have eigenvalues $\{\lambda_i\}_{i=1}^n$ and $B \in \mathbb{R}^{m\times m}$ have eigenvalues $\{\mu_j\}_{j=1}^m$.
Then the $mn$ eigenvalues of $A \otimes B$ are
\begin{equation*}
  \lambda_1\mu_1,\ldots,\lambda_1\mu_m, \lambda_2\mu_1,\ldots,\lambda_2\mu_m,\ldots,\lambda_n\mu_1\ldots,\lambda_n\mu_m.
\end{equation*}
\end{lemma}

\begin{theorem}\label{theorem2}
Let $\mathcal{A}_x$, $\mathcal{A}_y$ and  $\mathcal{A}_z$ be defined by (\ref{2.37}). Then
\begin{equation*}
  \begin{split}
   & ||(I-\mathcal{A}_\nu)^{-1}|| \leq 1, \\
   &  ||(I-\mathcal{A}_\nu)^{-1}(I+\mathcal{A}_\nu)|| \leq 1,
  \end{split}
\end{equation*}
where $\nu=x,y,z$.
\end{theorem}

\begin{proof}
  From Lemma \ref{lemma5} and (\ref{2.37}), we obtain
\begin{equation}\label{4.2}
  \begin{split}
&\frac{\mathcal{A}_x+\mathcal{A}_x^T}{2}=\frac{ (d_1^x+d_2^x)\tau}{2\Gamma(4-\alpha)(\Delta x)^{\alpha}}
       I \otimes I \otimes \left( \frac{ A_{\alpha}+ A_{\alpha}^T}{2} \right);\\
&\frac{\mathcal{A}_y+\mathcal{A}_y^T}{2}=\frac{ (d_1^y+d_2^y)\tau}{2\Gamma(4-\beta)(\Delta y)^{\beta}}
       I \otimes \left( \frac{ A_{\beta}+ A_{\beta}^T}{2} \right) \otimes  I ;\\
&\frac{\mathcal{A}_z+\mathcal{A}_z^T}{2}=\frac{ (d_1^z+d_2^z)\tau}{2\Gamma(4-\gamma)(\Delta z)^{\gamma}}
       \left( \frac{ A_{\gamma}+ A_{\gamma}^T}{2} \right) \otimes I \otimes  I.
  \end{split}
\end{equation}

According to Theorem \ref{theorem1} and Lemma \ref{lemma6}, we know that $\mathcal{A}_\nu+\mathcal{A}_\nu^T$
 are negative definite and symmetric matrices, where $\nu=x,y,z$. Then for any $v=(v_1,v_2,\ldots,v_n)^T \in \mathbb{R}^n$, we have
 \begin{equation*}
   v^Tv \leq v^T(I-\mathcal{A}_\nu^T)(I-\mathcal{A}_\nu)v.
 \end{equation*}
Substituting $v$ and $v^T$ by $(I-\mathcal{A}_\nu)^{-1}v$  and $v^T(I-\mathcal{A}_\nu^T)^{-1}$, respectively, leads to
 \begin{equation*}
 v^T(I-\mathcal{A}_\nu^T)^{-1}(I-\mathcal{A}_\nu)^{-1}v \leq v^Tv.
 \end{equation*}
Then, there exists
\begin{equation*}
  ||(I-\mathcal{A}_\nu)^{-1}||
  =\sup_{v\neq 0}\sqrt{\frac{v^T(I-\mathcal{A}_\nu^T)^{-1}(I-\mathcal{A}_\nu)^{-1}v}{v^Tv}}\leq 1.
\end{equation*}
 Similarly, we have
 \begin{equation*}
v^T(I+\mathcal{A}_\nu^T)(I+\mathcal{A}_\nu)v \leq v^T(I-\mathcal{A}_\nu^T)(I-\mathcal{A}_\nu)v.
 \end{equation*}
 Taking $v$ by $ (I-\mathcal{A}_\nu)^{-1}v$, then the above equation can be rewritten as
  \begin{equation*}
v^T(I-\mathcal{A}_\nu^T)^{-1}(I+\mathcal{A}_\nu^T)(I+\mathcal{A}_\nu)(I-\mathcal{A}_\nu)^{-1}v \leq v^Tv.
 \end{equation*}
 From Lemma \ref{lemma4}, it is to check that $\mathcal{A}_x$, $\mathcal{A}_y$ and $\mathcal{A}_z$ commute,
 then it yields that
  \begin{equation*}
  \begin{split}
  &||(I-\mathcal{A}_\nu)^{-1}(I+\mathcal{A}_\nu)|| =||(I+\mathcal{A}_\nu)(I-\mathcal{A}_\nu)^{-1}||\\
  &\quad=\sup_{v\neq0}\sqrt{\frac{v^T(I-\mathcal{A}_\nu^T)^{-1}(I+\mathcal{A}_\nu^T)
              (I+\mathcal{A}_\nu)(I-\mathcal{A}_\nu)^{-1}v}{v^Tv}}\leq 1.
  \end{split}
\end{equation*}

\end{proof}

\subsection{Stability and Convergence for 1D }

\begin{theorem}\label{theorem3}
The difference scheme (\ref{2.11})  with $\alpha \in (1,2)$
is unconditionally stable.
\end{theorem}
\begin{proof}
Let $\widetilde{u}_{i}^n~(i=1,2,\ldots,N_x-1;\,n=0,1,\ldots,N_t)$ be the approximate solution of $u_i^n$,
which is the exact solution of the difference scheme (\ref{2.11}).
Putting $\epsilon_i^n=\widetilde{u}_i^n-u_i^n$, and denoting $E^n=[\epsilon_1^n,\epsilon_2^n,\ldots, \epsilon_{N_x-1}^n]$,
 then from (\ref{2.11}) we obtain the following perturbation equation
 \begin{equation*}
  (I-M)E^{n+1}=(I+M)E^n,
\end{equation*}
where
\begin{equation}\label{4.3}
  M=  \frac{\tau }{2}\left(\frac{d_1^x}{\Gamma(4-\alpha)(\Delta x)^{\alpha}}A
   +\frac{ d_2^x}{\Gamma(4-\alpha)(\Delta x)^{\alpha}}A^T
                     +\frac{\kappa_x}{2\Delta x}B\right).
\end{equation}
Denoting $\lambda$ as an eigenvalue of the matrix $M$, and
using (\ref{4.3}), there exists
\begin{equation*}
\frac{M+M^T}{2}=\frac{\tau (d_1^x+d_2^x)}{2\Gamma(4-\alpha)(\Delta x)^{\alpha}}H,
\end{equation*}
 where $H$ is defined by (\ref{4.1}) and negative definite by the proof Theorem \ref{theorem1}, then from Lemma \ref{lemma3}, we get $\Re(\lambda(M))<0$.

Note that
 $\lambda$ is an eigenvalue of the matrix $M$ if and only if
$1-\lambda$ is an eigenvalue of the matrix $I-M$, if and only if $(1-\lambda)^{-1}(1+\lambda)$ is an
eigenvalue of the matrix $(I-M)^{-1}(I+M)$. Since $\Re(\lambda(M))<0 $, it implies that $|(1-\lambda)^{-1}(1+\lambda)|<1$.
Thus, the spectral radius of the matrix $(I-M)^{-1}(I+M)$ is less than $1$, hence the scheme (\ref{2.11})
is unconditionally stable.

\end{proof}

\begin{theorem}\label{theorem4}
Let $u(x_i,t_n)$ be the exact solution of (\ref{2.7})  with $\alpha \in (1,2)$, and  $u_i^n$ be the  solution of
the  finite difference scheme (\ref{2.11}),  then there is a positive constant C such that
\begin{equation*}
  \begin{split}
||u(x_i,t_n)-u_i^n|| \leq  C (\tau^2+(\Delta x)^2), \quad i=1,2,\ldots,N_x-1;\,n=0,1,\ldots,N_t.
  \end{split}
  \end{equation*}
\end{theorem}
\begin{proof}
Denoting $e_i^n=u(x_i,t_n)-u_i^n$, and    $e^n=[e_1^n,e_2^n,\ldots, e_{N_x-1}^n]^T$.
 Subtracting (\ref{2.9}) from (\ref{2.11}) and using $e^0=0$, we obtain
 \begin{equation*}
  (I-M)e^{n+1}=(I+M)e^n+R^{n+1},
\end{equation*}
where $M$ is defined by (\ref{4.3}), and $R^n=[R_1^n,R_2^n,\ldots, R_{N_x-1}^n]^T$. The above equation can be rewritten as
  \begin{equation*}
  e^{n+1}=(I-M)^{-1}(I+M)e^n+(I-M)^{-1}R^{n+1},
\end{equation*}
 and taking the 2-norm on both sides, similar to the proof of the Theorem \ref{theorem2}, we can show that $||(I-M)^{-1}(I+M)||\leq 1$.
 Then, using  $|R_i^{n+1}|\leq \widetilde{c} \tau(\tau^2+(\Delta x)^2)$ in (\ref{2.10}), we obtain
   \begin{equation*}
  ||e^{n}||\leq ||(I-M)^{-1}(I+M)e^{n-1}||+ ||R^{n}|| \leq  ||e^{n-1}||+||R^{n}||\leq\sum_{k=0}^{n-1}||R^{k+1}||\leq
  c (\tau^2+(\Delta x)^2).
\end{equation*}
\end{proof}

\subsection{Stability and Convergence for 2D }

\begin{theorem}\label{theorem3}
The difference scheme (\ref{2.21}) with $\alpha,\beta \in (1,2)$
is unconditionally stable.
\end{theorem}
\begin{proof}
Let $\widetilde{u}_{i,j}^n~(i=1,2,\ldots,N_x-1;j=1,2,\ldots,N_y-1;n=0,1,\ldots,N_t)$ be the approximate solution of ${u}_{i,j}^n$,
which is the exact solution of the  difference scheme (\ref{2.21}).
Taking $\epsilon_{i,j}^n=\widetilde{u}_{i,j}^n-{u}_{i,j}^n$,
 then from (\ref{2.21}) we obtain the following perturbation equation
 \begin{equation}\label{4.4}
  (I-\mathcal{B}_x )(I-\mathcal{B}_y )\mathbf{E}^{n+1}= (I+\mathcal{B}_x )(I+\mathcal{B}_y )\mathbf{E}^{n},
\end{equation}
where $\mathcal{B}_x$ and  $\mathcal{B}_y$ are given in (\ref{2.27}), and
\begin{equation*}
\begin{split}
&\mathbf{E}^{n}=[\epsilon_{1,1}^n,\epsilon_{2,1}^n,\ldots,\epsilon_{N_x-1,1}^n,
   \epsilon_{1,2}^n,\epsilon_{2,2}^n,\dots,\epsilon_{N_x-1,2}^n,\ldots,\epsilon_{1,N_y-1}^n,\epsilon_{2,N_y-1}^n,\ldots,\epsilon_{N_x-1,N_y-1}^n]^T,
\end{split}
\end{equation*}
and we  can write (\ref{4.4}) as the following form
 \begin{equation}\label{4.5}
  \mathbf{E}^{n+1}= (I-\mathcal{B}_y )^{-1}(I-\mathcal{B}_x )^{-1}(I+\mathcal{B}_x )(I+\mathcal{B}_y )\mathbf{E}^{n}.
\end{equation}
Using Lemma \ref{lemma4}, it is to check that $\mathcal{B}_x$ and $\mathcal{B}_y$ commute, i.e.,
\begin{equation}\label{4.6}
\begin{split}
\mathcal{B}_x\mathcal{B}_y=\mathcal{B}_y\mathcal{B}_x  =
&\left(\frac{ d_1^y\tau}{2\Gamma(4-\beta)(\Delta y)^{\beta}} A_{\beta}
               +\frac{  d_2^y\tau}{2\Gamma(4-\beta)(\Delta y)^{\beta}}  A_{\beta}^T  +\frac{\kappa_y\tau }{4\Delta y}  B \right)\\
&\quad \otimes \left(\frac{ d_1^x\tau}{2\Gamma(4-\alpha)(\Delta x)^{\alpha}} A_{\alpha}
               +\frac{  d_2^x\tau}{2\Gamma(4-\alpha)(\Delta x)^{\alpha}} A_{\alpha}^T +\frac{\kappa_x\tau }{4\Delta x} B\right).
\end{split}
 \end{equation}
 Then Eq. (\ref{4.5}) can be rewritten as
  \begin{equation*}
  \mathbf{E}^{n}= \left((I-\mathcal{B}_y )^{-1}(I+\mathcal{B}_y )\right)^n \left(((I-\mathcal{B}_x )^{-1}(I+\mathcal{B}_x )\right)^n\mathbf{E}^{0}.
\end{equation*}
From Lemma \ref{lemma5} and (\ref{2.27}), we obtain
\begin{equation*}
  \begin{split}
&\frac{\mathcal{B}_x+\mathcal{B}_x^T}{2}=\frac{ (d_1^x+d_2^x)\tau}{2\Gamma(4-\alpha)(\Delta x)^{\alpha}}
       I \otimes \left( \frac{ A_{\alpha}+ A_{\alpha}^T}{2} \right);\\
&\frac{\mathcal{B}_y+\mathcal{B}_y^T}{2}=\frac{ (d_1^y+d_2^y)\tau}{2\Gamma(4-\beta)(\Delta y)^{\beta}}
       \left( \frac{ A_{\beta}+ A_{\beta}^T}{2} \right) \otimes  I .\\
  \end{split}
\end{equation*}
 According to Theorem \ref{theorem1} and Lemma \ref{lemma2}, the eigenvalues of $\frac{A_{\alpha}+A_{\alpha}^T}{2}$ and $\frac{A_{\beta}+A_{\beta}^T}{2}$ are all negative
 when $\alpha, \beta \in (1,2)$. Let $\lambda_x$ and $\lambda_y$ be an eigenvalue of matrices $\mathcal{B}_x$ and $\mathcal{B}_y$, respectively.
From Lemma \ref{lemma6}, we get $\Re(\lambda_x)<0$ and $\Re(\lambda_y)<0$, then, the
eigenvalues of the matrices $(I-\mathcal{B}_x)^{-1}(I+\mathcal{B}_x)$ and $(I-\mathcal{B}_y)^{-1}(I+\mathcal{B}_y)$, $(1-\lambda_x)^{-1}(1+\lambda_x)$ and $(1-\lambda_y)^{-1}(1+\lambda_y)$ are less than one. 
And it follows that
$\left((I-\mathcal{B}_y )^{-1}(I+\mathcal{B}_y )\right)^n$ and $ \left(((I-\mathcal{B}_x )^{-1}(I+\mathcal{B}_x )\right)^n$
converge to zero matrix \cite[p.\,26]{Quarteroni:07}, as $n\rightarrow \infty$. Hence the difference scheme (\ref{2.19})  is unconditionally stable.

\end{proof}

\begin{theorem}\label{theorem4}
Let $u(x_i,y_j,t_n)$ be the exact solution of (\ref{2.15}) with $\alpha,\beta \in (1,2)$, and
 $u_{i,j}^n$ be the  solution of
the  finite difference scheme (\ref{2.21}),  then there is a positive constant C such that
\begin{equation*}
  \begin{split}
||u(x_i,y_j,t_n)-u_{i,j}^n|| \leq  C (\tau^2+(\Delta x)^2+(\Delta y)^2),
  \end{split}
  \end{equation*}
where
$i=1,2,\ldots,N_x-1;j=1,2,\ldots,N_y-1;\,n=0,1,\ldots,N_t.$
\end{theorem}
\begin{proof}
Taking  $e_{i,j}^n=u(x_i,y_j,t_n)-u_{i,j}^n$, and
 subtracting (\ref{2.16}) from (\ref{2.21}),  we obtain
 \begin{equation}\label{4.7}
  (I-\mathcal{B}_x )(I-\mathcal{B}_y )\mathbf{e}^{n+1}= (I+\mathcal{B}_x )(I+\mathcal{B}_y )\mathbf{e}^{n}+\mathbf{R}^{n+1},
\end{equation}
where $\mathcal{B}_x$ and  $\mathcal{B}_y$ are given in (\ref{2.27}), and
\begin{equation*}
\begin{split}
\mathbf{e}^{n}&=[e_{1,1}^n,e_{2,1}^n,\ldots,e_{N_x-1,1}^n,
   e_{1,2}^n,e_{2,2}^n,\dots,e_{N_x-1,2}^n,\ldots,e_{1,N_y-1}^n,e_{2,N_y-1}^n,\ldots,e_{N_x-1,N_y-1}^n]^T,\\
\mathbf{R}^{n}&=[R_{1,1}^n,R_{2,1}^n,\ldots,R_{N_x-1,1}^n,
   R_{1,2}^n,R_{2,2}^n,\dots,R_{N_x-1,2}^n,\ldots,R_{1,N_y-1}^n,R_{2,N_y-1}^n,\ldots,R_{N_x-1,N_y-1}^n]^T,\\
\end{split}
\end{equation*}
and  $|R_{i,j}^{n+1}|\leq \widetilde{c} \tau(\tau^2+(\Delta x)^2+(\Delta y)^2)$ is given in (\ref{2.17}).

 Since $\mathcal{B}_x$ and $\mathcal{B}_y$ commutes in (\ref{4.6}), then Eq. (\ref{4.7}) can be rewritten as
 \begin{equation*}
  \mathbf{e}^{n+1}= (I-\mathcal{B}_x )^{-1}(I+\mathcal{B}_x )(I-\mathcal{B}_y )^{-1}(I+\mathcal{B}_y )\mathbf{e}^{n}
  +(I-\mathcal{B}_y )^{-1}(I-\mathcal{B}_x )^{-1}\mathbf{R}^{n+1},
\end{equation*}
 and taking the 2-norm on both sides, similar to the proof of Theorem \ref{theorem2}, it can be proven that $ \|(I-\mathcal{B}_x )^{-1}(I+\mathcal{B}_x )(I-\mathcal{B}_y )^{-1}(I+\mathcal{B}_y )\| \leq \|(I-\mathcal{B}_x )^{-1}(I+\mathcal{B}_x )\|\cdot \|(I-\mathcal{B}_y )^{-1}(I+\mathcal{B}_y )\|\leq1 $  and
 $||(I-\mathcal{B}_y )^{-1}(I-\mathcal{B}_x )^{-1}|| \leq ||(I-\mathcal{B}_y )^{-1}||\cdot||(I-\mathcal{B}_x )^{-1}|| \leq 1 $.
 Then we get
   \begin{equation*}
  || \mathbf{e}^{n}||\leq \sum_{k=0}^{n-1}||\mathbf{R}^{k+1}||\leq
  c (\tau^2+(\Delta x)^2+(\Delta y)^2).
\end{equation*}
\end{proof}

By almost the same proof to the theorems of 2D, we can prove the following results for 3D.
\begin{theorem}
The difference scheme (\ref{2.33}) of the fractional convection diffusion equation (\ref{1.1}) with $\alpha,\beta, \gamma \in (1,2)$
is unconditionally stable.
\end{theorem}

\begin{theorem}
Let $u(x_i,y_j,z_m,t_n)$ be the exact solution of (\ref{1.1}) with $\alpha,\beta, \gamma \in (1,2)$,
and  $u_{i,j,m}^n$ be the  solution of
the  finite difference scheme (\ref{2.33}),  then there is a positive constant C such that
\begin{equation*}
  \begin{split}
||u(x_i,y_j,z_m,t_n)-u_{i,j,m}^n|| \leq  C (\tau^2+(\Delta x)^2+(\Delta y)^2+(\Delta z)^2),
  \end{split}
  \end{equation*}
where
$i=1,2,\ldots,N_x-1;j=1,2,\ldots,N_y-1;m=1,2,\ldots,N_z-1;n=0,1,\ldots,N_t.$
\end{theorem}

\subsection{Stability and Convergence of ($1.1^\prime$) by using D-ADI-II and FS-II }
To prove the stability and convergence of  D-ADI-II and FS-II for ($1.1^\prime$), we need the following two lemmas.
\begin{lemma}\cite[p.\,396]{Golub:96}\label{lemma7}
If $P$ and $P+Q$ are m-by-m symmetric matrices, then
\begin{equation*}
  \lambda_k(P)+ \lambda_m(Q) \leq \lambda_k(P+Q) \leq \lambda_k(P) + \lambda_1(Q), \quad k=1,2,\ldots,m,
\end{equation*}
with  eigenvalues $\lambda_1(\cdot)\geq \lambda_2(\cdot)\geq \cdots \geq\lambda_m(\cdot)$.
\end{lemma}

\begin{lemma}\cite[p.\,84]{Griffiths:10}\label{lemma8}
Let the quadratic equation be $\lambda^2-b\lambda+c=0$, where $b$ and $c$ are both real parameters, then all roots satisfy
$|\lambda| <1$ if and only if $|b|< 1+c < 2$.
\end{lemma}

\begin{theorem}\label{theorem7}
The difference scheme (\ref{3.7}) corresponding to two-dimensional case of  ($1.1^\prime$) with $\alpha,\beta \in (1,2)$
is unconditionally stable.
\end{theorem}
\begin{proof}
For the two-dimensional case of ($1.1^\prime$), Eq.
 (\ref{2.27}) has the following form
\begin{equation}\label{4.008}
\begin{split}
&\mathcal{B}_x =\frac{ d^x\tau}{2\Gamma(4-\alpha)(\Delta x)^{\alpha}}I \otimes (A_{\alpha}+A_{\alpha}^T)=\frac{ \tau}{2}\mathcal{\widetilde{B}}_x,
~~\mbox{where}~~\mathcal{\widetilde{B}}_x=\frac{ d^x}{\Gamma(4-\alpha)(\Delta x)^{\alpha}}I \otimes (A_{\alpha}+A_{\alpha}^T),\\
&\mathcal{B}_y =\frac{ d^y\tau}{2\Gamma(4-\beta)(\Delta y)^{\beta}} (A_{\beta}+A_{\beta}^T)  \otimes I=\frac{ \tau}{2}\mathcal{\widetilde{B}}_y,
~~\mbox{where}~~\mathcal{\widetilde{B}}_y=\frac{ d^y}{\Gamma(4-\beta)(\Delta y)^{\beta}} (A_{\beta}+A_{\beta}^T)  \otimes I,
\end{split}
 \end{equation}
 and  $\mathcal{\widetilde{B}}_x$ and $\mathcal{\widetilde{B}}_y$ commute, i.e.,
\begin{equation*}
\begin{split}
\mathcal{\widetilde{B}}_x\mathcal{\widetilde{B}}_y=\mathcal{\widetilde{B}}_y\mathcal{\widetilde{B}}_x
 = \frac{ d^x}{\Gamma(4-\alpha)(\Delta x)^{\alpha}} \cdot
\frac{ d^y}{\Gamma(4-\beta)(\Delta y)^{\beta}}(A_{\beta}+A_{\beta}^T)  \otimes (A_{\alpha}+A_{\alpha}^T).
\end{split}
 \end{equation*}

Let $\widetilde{u}_{i,j}^n~(i=1,2,\ldots,N_x-1;j=1,2,\ldots,N_y-1;n=0,1,\ldots,N_t)$ be the approximate solution of ${u}_{i,j}^n$,
which is the exact solution of the difference scheme (\ref{3.7}).
Taking $\epsilon_{i,j}^n=\widetilde{u}_{i,j}^n-{u}_{i,j}^n$,
 then from (\ref{3.7}) we obtain the following perturbation equation
 \begin{equation*}
  (I-\mathcal{B}_x )(I-\mathcal{B}_y )\mathbf{E}^{n+1}= (I+\mathcal{B}_x )(I+\mathcal{B}_y )\mathbf{E}^{n}
  +\mathcal{{B}}_x\mathcal{{B}}_y\mathbf{E}^{n}-\mathcal{{B}}_x\mathcal{{B}}_y\mathbf{E}^{n-1},
\end{equation*}
i.e.,
 \begin{equation}\label{4.9}
\mathbf{E}^{n+1}= (P+Q)\mathbf{E}^{n}-Q\mathbf{E}^{n-1},
\end{equation}
where
\begin{equation*}
\begin{split}
&\mathbf{E}^{n}=[\epsilon_{1,1}^n,\epsilon_{2,1}^n,\ldots,\epsilon_{N_x-1,1}^n,
   \epsilon_{1,2}^n,\epsilon_{2,2}^n,\dots,\epsilon_{N_x-1,2}^n,\ldots,\epsilon_{1,N_y-1}^n,\epsilon_{2,N_y-1}^n,\ldots,\epsilon_{N_x-1,N_y-1}^n]^T,
\end{split}
\end{equation*}
and
 \begin{equation*}
 \begin{split}
P=(I-\mathcal{B}_x )^{-1}(I+\mathcal{B}_x )(I-\mathcal{B}_y )^{-1}(I+\mathcal{B}_y ),~~~~
Q=(I-\mathcal{B}_x )^{-1} \mathcal{B}_x (I-\mathcal{B}_y )^{-1} \mathcal{B}_y.
\end{split}
\end{equation*}
Therefore, Eq. (\ref{4.9}) can be rewritten as

\begin{equation*}
  \mathbf{V}^{n+1}=\mathbf{M}\mathbf{V}^{n},
\end{equation*}
with
\begin{equation*}
\mathbf{V}^{n+1}=\left [ \begin{matrix}
   \mathbf{E}^{n+1}        \\
    \mathbf{E}^{n}        \\
 \end{matrix}
 \right ],
 \quad \mbox{and } \quad
\mathbf{M}=\left [ \begin{matrix}
P+Q          &-Q        \\
I            &0     \\
 \end{matrix}
 \right ].
\end{equation*}
From \cite[p.\,128]{Yu:04}, we know that the eigenvalues of $\mathbf{M}$ are the same as the eigenvalues of $\mathbf{L}$, where
\begin{equation*}
\mathbf{L}=\left [ \begin{matrix}
\lambda_k(P+Q)          &-\lambda_k(Q)        \\
1           &0     \\
 \end{matrix}
 \right ],
\end{equation*}
then the eigenvalues $\lambda$ of $\mathbf{M}$ satisfies
\begin{equation*}
  \lambda^2-\lambda_k(P+Q)  \lambda + \lambda_k(Q) =0, \quad k=1,2,\ldots,m \,\,(m=N_x-1 \cdot N_y-1).
\end{equation*}
Similar to the above proof, we know that $\mathcal{\widetilde{B}}_x$ and $\mathcal{\widetilde{B}}_y$ are negative definite and symmetric matrices,
and the matrix $\mathcal{\widetilde{B}}_x\mathcal{\widetilde{B}}_y$ or $\mathcal{\widetilde{B}}_y\mathcal{\widetilde{B}}_x$ is positive definite and symmetric,
it follows that $\lambda_k(P+Q)$ and $\lambda_k(Q)$ are real numbers, and we have
\begin{equation*}
  |\lambda_k(P)|<1 \quad \mbox{and } \quad  0<  \lambda_k(Q)<1,
\end{equation*}
According to Lemma \ref{lemma7} and \ref{lemma8}, we get
\begin{equation*}
  -1-\lambda_k(Q)<\lambda_k(Q)+\lambda_m(P)< \lambda_k(P+Q) <  \lambda_k(Q)+\lambda_1(P)<\lambda_k(Q)+1,
\end{equation*}
thus,  the difference scheme
is unconditionally stable.
\end{proof}

\begin{theorem}\label{theorem8}
Let $u(x_i,y_j,t_n)$ be the exact solution of (\ref{2.15}) corresponding to two-dimensional  case of ($1.1^\prime$) with $\alpha,\beta \in (1,2)$, and
 $u_{i,j}^n$ be the  solution of
the  finite difference scheme (\ref{3.7}),  then there are a positive constant $C$ and some kind of norm $\||\cdot|\|$ such that
\begin{equation*}
  \begin{split}
\||u(x_i,y_j,t_n)-u_{i,j}^n\|| \leq  C (\tau^2+(\Delta x)^2+(\Delta y)^2),
  \end{split}
  \end{equation*}
where
$i=1,2,\ldots,N_x-1;j=1,2,\ldots,N_y-1;\,n=0,1,\ldots,N_t.$
\end{theorem}

\begin{proof}
For the two-dimensional case of ($1.1^\prime$),
 taking $e_{i,j}^n=u(x_i,y_j,t_n)-u_{i,j}^n$, 
 from (\ref{2.16}) and (\ref{3.7}),  we obtain
 \begin{equation}\label{4.10}
  (I-\mathcal{B}_x )(I-\mathcal{B}_y )\mathbf{e}^{n+1}= (I+\mathcal{B}_x )(I+\mathcal{B}_y )\mathbf{e}^{n}
  +\mathcal{B}_x\mathcal{B}_y(\mathbf{e}^{n}-\mathbf{e}^{n-1})+R^{n+1},
\end{equation}
where $\mathcal{B}_x$ and  $\mathcal{B}_y$ are given in (\ref{4.008}), and
\begin{equation*}
\begin{split}
\mathbf{e}^{n}&=[e_{1,1}^n,e_{2,1}^n,\ldots,e_{N_x-1,1}^n,
   e_{1,2}^n,e_{2,2}^n,\dots,e_{N_x-1,2}^n,\ldots,e_{1,N_y-1}^n,e_{2,N_y-1}^n,\ldots,e_{N_x-1,N_y-1}^n]^T,\\
R^{n}&=[R_{1,1}^n,R_{2,1}^n,\ldots,R_{N_x-1,1}^n,
   R_{1,2}^n,R_{2,2}^n,\dots,R_{N_x-1,2}^n,\ldots,R_{1,N_y-1}^n,R_{2,N_y-1}^n,\ldots,R_{N_x-1,N_y-1}^n]^T,\\
\end{split}
\end{equation*}
and  $|R_{i,j}^{n+1}|\leq \widetilde{c} \tau(\tau^2+(\Delta x)^2+(\Delta y)^2)$ is given in (\ref{2.17}). Similarly, taking
 \begin{equation*}
 \begin{split}
&P=(I-\mathcal{B}_x )^{-1}(I+\mathcal{B}_x )(I-\mathcal{B}_y )^{-1}(I+\mathcal{B}_y );\\
&Q=(I-\mathcal{B}_x )^{-1} \mathcal{B}_x (I-\mathcal{B}_y )^{-1} \mathcal{B}_y;\\ ~~
&S=(I-\mathcal{B}_x )^{-1} (I-\mathcal{B}_y )^{-1},
\end{split}
\end{equation*}
then, Eq. (\ref{4.10}) can be rewritten as
\begin{equation}\label{4.11}
  \mathbf{V}^{n+1}=\mathbf{M}\mathbf{V}^{n}+\mathbf{N}\mathbf{R}^{n+1},
\end{equation}
with
\begin{equation*}
\mathbf{V}^{n+1}=\left [ \begin{matrix}
   \mathbf{e}^{n+1}        \\
    \mathbf{e}^{n}        \\
 \end{matrix}
 \right ],
 ~~
 \mathbf{R}^{n+1}=\left [ \begin{matrix}
   {R}^{n+1}        \\
    {R}^{n}        \\
 \end{matrix}
 \right ],
~~
\mathbf{M}=\left [ \begin{matrix}
P+Q          &-Q        \\
I            &0     \\
 \end{matrix}
 \right ],
  \quad \mbox{and } \quad
\mathbf{N}=\left [ \begin{matrix}
S          &0        \\
0          &0     \\
 \end{matrix}
 \right ].
 \end{equation*}
Similarly, we can prove $|\lambda_k(\mathbf{N})|<1$, then there exists some kind of norm $\||\cdot|\|$ such that $\||\mathbf{M}|\| \leq 1$, and $\||\mathbf{N}\||\leq 1$.
Taking the norm on both sides of (\ref{4.11}) 
leads to
   \begin{equation*}
  \|| \mathbf{e}^{n}\|| \leq \|| \mathbf{V}^{n}\||\leq \sum_{k=0}^{n-1}\||\mathbf{R}^{k+1}\||\leq
  c (\tau^2+(\Delta x)^2+(\Delta y)^2).
\end{equation*}
\end{proof}

All the theoretical results for the three-dimensional case $(1.1^\prime)$ can be obtained by the same way of the two-dimensional case of $(1.1^\prime)$. For the briefness of the paper, we omit them here.

\section{Numerical Results}
Here we verify the above theoretical results including convergent
order and stability. Introducing
the vectors $U(\Delta x)=[u_h(x_0,t),\ldots,u_h(x_n,t)]^{\rm T}$,
where $U$ is the approximated value,  and $u(\Delta
x)=[u(x_0,t),\ldots,u(x_n,t)]^{\rm T}$, where $u$ is the exact value
and the stepsize in space is $\Delta x$, i.e., $\Delta
x=x_{i+1}-x_i$, in the following numerical examples the errors are
measured by
\begin{equation}\label{5.1}
  ||U(\Delta x)-u(\Delta x)||_\infty,
\end{equation}
where $||\cdot||_\infty$ is the maximum norm.
\subsection{Numerical results for 1D}

Let us consider the one-dimensional two-sided fractional convection diffusion
equation (\ref{2.7}), where $0< x < 1 $ and $0 < t \leq 1$, with the coefficients $d_1^x=d_2^x=\kappa_x=1$, and the forcing function
\begin{equation*}
\begin{split}
 f(x,t)&=-e^{-t} \Big[  x^2 \left(1-x \right)^2+ \left(4x^3-6x^2+2x \right) +\frac{\Gamma(3)}{\Gamma(3-\alpha)}\left(x^{2-\alpha}+(1-x)^{2-\alpha}\right) \\
       &\qquad -2\frac{\Gamma(4)}{\Gamma(4-\alpha)}\left(x^{3-\alpha}+(1-x)^{3-\alpha}\right)
        +\frac{\Gamma(5)}{\Gamma(5-\alpha)}\left(x^{4-\alpha}+(1-x)^{4-\alpha} \right)   \Big],
  \end{split}
\end{equation*}
the initial condition $u(x,0)=x^2(1-x)^2$, the boundary conditions $u(0,t)=u(1,t)=0$,
and the exact solution of the equation is $u(x,t)=e^{-t}x^2(1-x)^2$.

\begin{table}[h]\fontsize{9.5pt}{12pt}\selectfont
  \begin{center}
  \caption {The maximum errors (\ref{5.1}) and convergent orders for the scheme (\ref{2.11}) of the one-dimensional two-side fractional convection diffusion equation (\ref{2.7}) at t=1 and $\Delta t=\Delta x$.} \vspace{5pt}
\begin{tabular*}{\linewidth}{@{\extracolsep{\fill}}*{7}{c}}                                    \hline  
         $\Delta t,\,\Delta x$ & $\alpha=1.1$ & Rate   & $\alpha=1.5$ & Rate   & $\alpha=1.9$ &   Rate      \\\hline
              ~~1/10&   0.0022&          &  0.0011&               &   0.0010&             \\\hline 
             ~~1/20&   4.5729e-004&  2.2916&  2.6284e-004&   2.0596&  2.5502e-004&   1.9917    \\\hline 
             ~~1/40&   1.0712e-004&  2.0939&  6.2954e-005&  2.0618&  6.4257e-005&   1.9887   \\\hline 
             ~~1/80&   2.5242e-005&  2.0853&  1.5067e-005&  2.0628&  1.6169e-005&   1.9906    \\\hline 
             ~~1/160&   5.9414e-006&  2.0869&  3.6083e-006&  2.0620&  4.0594e-006&  1.9939   \\\hline 
    \end{tabular*}\label{tab:a}
  \end{center}
\end{table}
In Table \ref{tab:a}, we show that the scheme  (\ref{2.11}) is second order convergent in both space and time.
\subsection{Numerical results for 2D}
 Consider the two-dimensional two-sided space fractional convection diffusion equation (\ref{2.15}),
on a finite domain $ 0<x<2,\,0<y<2$, $0< t \leq 2$, and with the
coefficients
 \begin{equation*}
\begin{split}
&d_1^x=\Gamma(3-\alpha)x^{\alpha},\quad d_2^x =\Gamma(3-\alpha)(2-x)^{\alpha},\quad \kappa_x=\frac{1}{4}x,\\
&d_1^y=\Gamma(3-\beta)y^{\beta},\quad d_2^y =\Gamma(3-\beta)(2-y)^{\beta}, \,\quad \kappa_y=\frac{1}{4}y,\\
\end{split}
\end{equation*}
and the forcing function
\begin{equation*}
\begin{split}
 f(x,y,t)=&-4e^{-t}x^2y^2(x-2)(y-2)(3xy-5x-5y+8)\\
  &-32e^{-t}y^2(2-y)^2\left[ x^2+(2-x)^2- \frac{3\left( x^3+(2-x)^3  \right)}{3-\alpha}
  +\frac{3\left(x^4+(2-x)^4  \right)}{(3-\alpha)(4-\alpha)}   \right] \\
  &-32e^{-t}x^2(2-x)^2\left[ y^2+(2-y)^2- \frac{3\left( y^3+(2-y)^3  \right)}{3-\beta}
  +\frac{3\left(y^4+(2-y)^4  \right)}{(3-\beta)(4-\beta)}   \right],
  \end{split}
\end{equation*}
and the initial condition $u(x,y,0)=4x^2(2-x)^2 y^2(2-y)^2$ and the
Dirichlet boundary conditions on the rectangle in the form $
u(0,y,t)=u(x,0,t)=0$ and $ u(2,y,t)=u(x,2,t)=0$
 for all $ t>0$. The exact solution to this two-dimensional two-sided fractional convection diffusion equation is
 \begin{equation*}
\begin{split}
u(x,y,t)=4e^{-t}x^2(2-x)^2 y^2(2-y)^2.
  \end{split}
  \end{equation*}

\begin{table}  [h]  \fontsize{9.5pt}{12pt}  \selectfont
  \begin{center}
  \caption{
  The maximum errors (\ref{5.1}) and convergent orders for the scheme (\ref{2.25})-(\ref{2.26}) of the two-dimensional two-sided fractional convection diffusion equation (\ref{2.15}) at $t=2$ with $\tau
=\Delta x =\Delta y$.}\vspace{5pt}

    \begin{tabular*}{\linewidth}{@{\extracolsep{\fill}}*{7}{l|c|c|c|c|c|r}}                                   \hline
         $\tau,\,\Delta x,\,\Delta y$ & $\alpha=1.1,\beta=1.2$ & Rate   & $\alpha=1.5,\beta=1.4$ & Rate   & $\alpha=1.9,\beta=1.9$ &   Rate      \\\hline
              ~~1/10&                  0.0133&        &                 0.0116 &        &                  0.0109&             \\\hline
              ~ 1/20&                  0.0033&  1.9966&                  0.0029&  1.9830&                  0.0028&  1.9750\\\hline
              ~~1/40&             8.3408e-004&  1.9985&             7.4001e-004&  1.9876&             7.0378e-004&  1.9739   \\\hline
              ~~1/80&             2.0877e-004&  1.9982&             1.8612e-004&  1.9913&             1.7900e-004&  1.9752   \\\hline
             ~~1/160&             5.2231e-005&  1.9990&             4.6726e-005&  1.9940&             4.5468e-005&  1.9770  \\\hline
    \end{tabular*}{\label{004}}
  \end{center}
\end{table}
Comparing Table \ref{004} with Table 2 of \cite{Chen:12}, we further
confirm that the PR-ADI and D-ADI are equivalent for solving
two-dimensional equations, since they have the completely same
maximum error values. Table \ref{004} numerically shows that the
D-ADI scheme (\ref{2.25})-(\ref{2.26}) is second order convergent
and this is in agreement with the order of the truncation error.

\subsection{Numerical results for 3D}
 Consider the three-dimensional two-sided fractional convection diffusion equation (\ref{1.1}),
on a finite domain $0<x<2$, $0<y<2$, $0<z<2$, $0< t \leq 2$, and with the coefficients
 \begin{equation*}
\begin{split}
&d_1^x=\Gamma(3-\alpha)x^{\alpha},\quad d_2^x=\Gamma(3-\alpha)(2-x)^{\alpha},\quad \kappa_x=\frac{1}{4}x,\\
&d_1^y=\Gamma(3-\beta)y^{\beta}, \quad d_2^y=\Gamma(3-\beta)(2-y)^{\beta}, \,\quad \kappa_y=\frac{1}{4}y,\\
&d_1^z=\Gamma(3-\gamma)z^{\gamma}, \quad d_2^z=\Gamma(3-\gamma)(2-z)^{\gamma}, \,\quad \kappa_z=\frac{1}{4}z,
\end{split}
\end{equation*}
 and the  zero Dirichlet boundary conditions on the cube  for all $t>0$,  the exact solution to this three-dimensional two-sided fractional convection diffusion equation is
 \begin{equation*}
\begin{split}
u(x,y,z,t)=4e^{-t}x^2(2-x)^2 y^2(2-y)^2z^2(2-z)^2.
  \end{split}
  \end{equation*}
According to  the above conditions, it is easy to get the forcing function $f(x,y,z,t)$.

  \begin{table}[h]\fontsize{9.5pt}{12pt}\selectfont
\begin{center}
 \caption{The maximum errors (\ref{5.1}) and convergent orders for the scheme (\ref{2.34})-(\ref{2.36}) of the
  three-dimensional two-sided fractional convection diffusion equation (\ref{1.1}) at $t=2$
  with $\tau=\Delta x =\Delta y=\Delta z$.}\vspace{5pt}

\begin{tabular*}{\linewidth}{@{\extracolsep{\fill}}*{7}{l|c|c|c|c|c|r}}                                   \hline
$~~\tau$ & $\alpha=\beta=\gamma=1.2$ & Rate & $\alpha=1.4,\beta=1.5,\gamma=1.6$ & Rate  & $\alpha=\beta=\gamma=1.9$ &Rate  \\\hline

       ~~1/10&             1.2063e-002&        &             1.3349e-002&        &             1.5558e-002&             \\\hline
       ~ 1/20&             3.0047e-003&  2.0053&             3.3242e-003&  2.0057&             3.7859e-003&  2.0390\\\hline
       ~~1/40&             7.5079e-004&  2.0008&             8.3225e-004&  1.9979&             9.4168e-004&  2.0073   \\\hline
       ~~1/80&             1.8773e-004&  1.9997&             2.0875e-004&  1.9952&             2.3625e-004&  1.9949  \\\hline
\end{tabular*}{\label{0006}}

\end{center}
\end{table}
Table \ref{0006} also shows the  maximum error, at  time $t=2$ and $\tau=\Delta x =\Delta y=\Delta z$,
between the exact analytical value and the
numerical value obtained by applying the D-ADI scheme
(\ref{2.34})-(\ref{2.36}), and the scheme is second order convergent
and this is in agreement with the order of the truncation error.

\section{Conclusions}
This work provides an algorithm which can efficiently solve
three-dimensional space fractional PDEs. The idea is to solve higher
dimensional problem by the strategy of dimension by dimension. When
realizing the idea, the splitting errors may be introduced, so the
techniques of diminishing the influences of splitting errors are
also discussed. The effectiveness of the algorithm is theoretically
proved and numerically verified.

%

%

\section*{Acknowledgments}This work was supported by the Program for
New Century Excellent Talents in University under Grant No.
NCET-09-0438, the National Natural Science Foundation of China under
Grant No. 10801067 and No. 11271173, and the Fundamental Research Funds for the
Central Universities under Grant No. lzujbky-2010-63 and No. lzujbky-2012-k26.


\end{document}